\documentclass[12pt, reqno]{amsart}
\usepackage[a4paper,vmargin={3.5cm,3.5cm},hmargin={2.5cm,2.5cm},bottom=30mm]{geometry}
\usepackage{amsfonts}
\usepackage{amsthm,amsmath,mathtools}
\usepackage{amsbsy}
\usepackage{bbm}
\usepackage{amssymb} 
\usepackage[nice]{nicefrac}
\usepackage{enumerate,enumitem}
\usepackage{color}
\usepackage[final]{showkeys}
\usepackage[utf8]{inputenc}
\usepackage[T1]{fontenc}  

\usepackage{graphicx} 
\usepackage{pgf,tikz}
\usetikzlibrary{arrows}

\usepackage{cmbright}
\usepackage[T1]{fontenc}

\usepackage{hyperref}
\hypersetup{ colorlinks=true, linkcolor=blue, citecolor=blue,
  filecolor=blue, urlcolor=blue}
  \numberwithin{equation}{section} 
  \date{}

\usepackage[
backend=biber,
style=alphabetic,
sorting=nyt,
doi=false,
isbn=false,
giveninits=true
]{biblatex}
\DeclareMathOperator{\R}{\mathbb{R}}
\DeclareMathOperator{\Q}{\mathbb{Q}}
\DeclareMathOperator{\Z}{\mathbb{Z}}
\DeclareMathOperator{\N}{\mathbb{N}}
\DeclareMathOperator{\C}{\mathbb{C}}

\DeclareMathOperator{\De}{d}
\DeclareMathOperator{\cyp}{cyp}
\renewcommand{\c}{\mathrm{c}}
\newcommand{\e}{\mathrm{e}}
\newcommand{\f}{\frac}  
\newcommand{\eps}{\varepsilon}
\newcommand{\vr}{\varphi}

\newcommand{\bxi}{\overline{\xi}}
\newcommand{\iu}{{\mathrm{i}\mkern1mu}} 

\newcommand{\cA}{\mathcal{A}}
\newcommand{\F}{\mathbb{F}}
\newcommand{\GA}{\mathcal{GA}}
\newcommand{\D}{\mathcal{D}}
\newcommand{\boxi}{\boldsymbol{\xi}}
\newcommand{\boxbxi}{\boldsymbol{\overline{\xi}}}
\newcommand{\boxzi}{\bf Z}
\newcommand{\boxbzi}{\overline{\bf Z}}
\newcommand{\Diff}{\D(\boxi,\,\boxbxi)}
\newcommand{\Diffz}{\D(\boxzi,\,\boxbzi)}
\newcommand{\boxx}{\boldsymbol{x}}
\newcommand{\boxy}{\boldsymbol{y}}
\newcommand{\prob}{\mathbf{P}}
\newcommand{\E}{\mathbf{E}} 
\newcommand{\Exp}{\mathrm{E}} 

\theoremstyle{plain} 
    \newtheorem{theorem}{Theorem}
    \numberwithin{theorem}{section} 
    \newtheorem{lemma}[theorem]{Lemma}
    \newtheorem{proposition}[theorem]{Proposition}
    \newtheorem{corollary}[theorem]{Corollary}

\theoremstyle{definition} 
    \newtheorem{definition}[theorem]{Definition}

    \newtheorem{remark}[theorem]{Remark}
    \newtheorem{example}[theorem]{Example}

\newcommand{\mat}[1]{\left(\begin{matrix}#1\end{matrix} \right)} 

 \makeatletter
 \def\paragraph{\@startsection{paragraph}{4}%
 \z@\z@{-\fontdimen2\font}%
   {\normalfont\itshape}}\makeatother
\title{Grassmannian calculus for probability}
\author[S. Baldassarri]{Simone Baldassarri}
\address{Leiden University, Mathematical Institute, Einsteinweg 55, 2333 CC Leiden, The Netherlands}
\email{s.baldassarri@math.leidenuniv.nl}
\thanks{}
\author[A. Cipriani]{Alessandra Cipriani}
\address{UCL, Department of Statistical Science, 1-19 Torrington Place, London, WC1E 7HB, United Kingdom}
\email{a.cipriani@ucl.ac.uk}
\date{December 2024}

\begin{document}

\begin{abstract}
   The present overview and gentle introduction to Grassmannian calculus and some of its applications to probability collects the
notes of a mini-course given by the authors at the Brazilian School of Probability, August 5-9, 2024, in Salvador, Bahia, Brazil. The content is by no means comprehensive, and is a personal summary and interpretation of results and applications of this interesting area of research. 
\end{abstract}
\maketitle

\section{Introduction}

Grassmannian variables, originating from the work of Hermann Grassmann~\cite{lewis2005hermann}, form a critical part of modern mathematics, particularly in algebraic geometry, combinatorics, and physics. In the early 20th century, Grassmannian variables began to see applications beyond pure geometry, finding a place in quantum field theory and theoretical physics. In physics terms, they provide a mathematical framework to accommodate variables satisfying Pauli exclusion principle. Therefore they are crucial for describing systems with both bosonic (commuting) and fermionic (anticommuting) components. Their role in supersymmetry became crucial, and this development laid the groundwork for their eventual integration into probabilistic and combinatorial models, where their algebraic properties could be used to describe complex processes. These applications leveraged the unique properties of Grassmann calculus, especially the way in which ``it permits the expression in formulas of the results of geometric constructions''~\cite{peano1999geometric}. Grassmannian structures emerge in the analysis of determinants, particularly when using algebraic methods to compute the probability distributions associated with spanning trees. By representing the set of all possible spanning trees as points in a Grassmannian, one can explore the geometric and algebraic properties of these trees, such as their symmetries and invariants. This approach also extends to the study of forests (disconnected trees) and uniform spanning forests, which generalize the concept of uniform spanning trees (USTs) to infinite graphs. 

These notes are primarily concerned with the application of Grassmann calculus to stochastic models related to USTs, however it is worthwhile mentioning that other areas in which Grassmannian calculus has been applied range from random matrix theory to lattice models, such as dimer models and the Ising model, and renormalization theory~\cite{bauerschmidt2019introduction,wegner2016supermathematics}, just to name a few.

\section{A motivation: the Abelian sandpile model}
One of the intriguing applications of Grassman calculus lies in its link to combinatorial models such as the aforementioned USTs and the Abelian Sandpile Model (ASM). They are both deeply connected to self-organized criticality and the broader study of complex random systems.

The ASM is also known as the Bak-Tang-Wiesenfeld model~\cite{bak1987self, 10.1214/14-PS228,redig2006mathematical}. It is a type of cellular automaton defined on a graph, typically a grid, where each cell (or vertex) can hold a certain number of grains of sand. In this model, height fields refer to the number of sand grains at each vertex of the graph. Each vertex $i$ has an associated height $h_i$, which is an integer representing the number of grains at that vertex. The configuration of the entire system is given by the set of heights at all vertices. The model follows a dynamics in three steps. Firstly, grains of sand are added one at a time to randomly chosen vertices. Secondly, topplings may occur. If the height at any vertex exceeds a certain threshold $k_{\text{thresh}}$, that vertex topples, distributing one grain of sand to each of its neighboring vertices. This can cause neighboring vertices to exceed their thresholds and topple as well, leading to a cascade of topplings, known as an avalanche. Finally, the sandpile stabilizes, meaning this process continues until all vertices are below the threshold, resulting in a stable configuration. 

One of the key features of the Abelian sandpile model is its Abelian property. This means that the final stable configuration of the sandpile does not depend on the order in which the grains are added or the order in which the vertices topple. This property simplifies the analysis and allows for exact results in many cases. Despite this and its apparently simple description, studying the ASM presents several notable challenges. Firstly, the system exhibits multifractal scaling rather than simple finite-size scaling. This multifractality is a hallmark of systems exhibiting self-organized criticality. Additionally, the height fields exhibit elaborate, non-local correlations. While the height-one variables can be handled by local calculations thanks to the burning bijection~\cite{Majumdar1992Jun}, higher height variables involve more intricate interactions. Finally,
in 2D large avalanches, though rare, dominate the statistics in the thermodynamic limit. These rare events can significantly affect the height field distributions and their correlations. These factors make the study of height fields in the Abelian sandpile model a rich and challenging area of research.

The burning bijection relates the stationary measure of the ASM on a graph to the uniform spanning tree measure of the same graph. A spanning tree of a finite connected graph $G$ is a subgraph which has no loops and connects via its edges all points of $G$. Kirchhoff's theorem gives the number of spanning trees in this setup, giving rise to a uniform probability measure on all such trees, called the uniform spanning tree. This allows one to give an alternative description of many observables of the ASM. For example, for a suitable collection of vertices $V$ the event $\{\deg_{\text{UST}}(v)=1, \forall v \in  V\}$ is equivalent to $\{h_v=1, \forall v \in V\}$, and that the UST is incident to $v$ via a preferred edge $\eta(v)$ (see Figure~\ref{fig:UST_att}). 
\begin{figure}
    \centering
\includegraphics[width=0.5\linewidth]{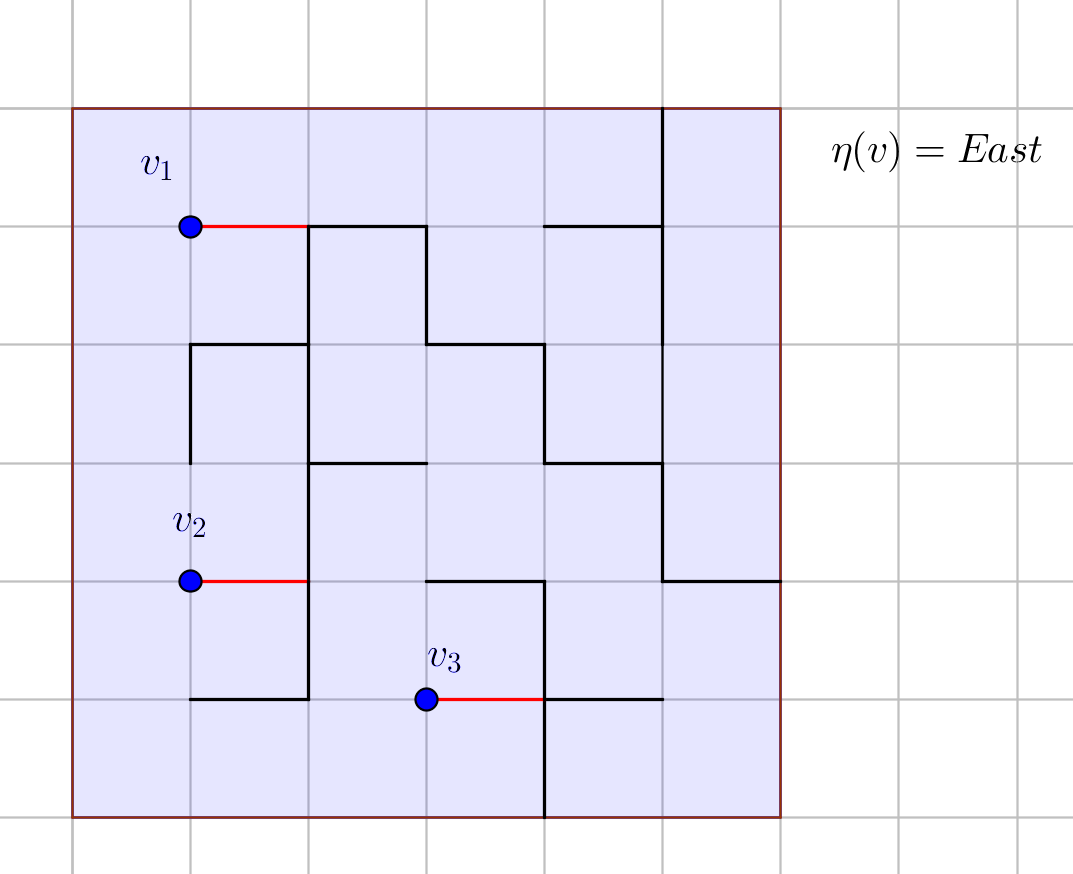}
    \caption{The event $\displaystyle h(v_1)=h(v_2)=h(v_3)=1$ in ASM is equivalent to the UST incident to $v_i$ via the east neighbor.}
    \label{fig:UST_att}
\end{figure}
This means that the height-one field is a local event, that is, an event which is measurable with respect to a finite number of edges only, even as we let the size of the UST grow. This makes it amenable to exact computations than higher heights, which in contrast are non-local~\cite{10.1214/14-PS228}.
The perhaps surprising fact is that this 0-1 field is closely related to a special Gaussian field: the {\em discrete Gaussian free field} (DGFF).

\subsection{ASM and DGFF} In~\cite{durre,cipriani2022properties} an interesting relation between the height-one field of the ASM and the discrete Gaussian free field was unveiled. In order to explain it, we first need to introduce the DGFF~\cite[Chapter 2]{Szn}. The DGFF is a Gaussian vector indexed over a finite connected graph. In~\cite{durre,cipriani2022properties}, this graph is a subset of $\Z^d$, although extensions to many other graphs are possible. Let us call this graph $G=(V,\,E)$ with its vertex set $V$ and edge set $E$. On it, we define the Laplacian matrix $\Delta$ as follows: for $i,\,j\in V$
$$\Delta(i,\,j)=\begin{cases}
1& \text{if } \|i-j\|=1,\\
-2d & \text{if } i=j,\\
0 &\text{otherwise.}
\end{cases}$$
Finally, we identify the boundary of the graph $G$ as a subset $\emptyset\neq C\subset V$ of vertices. With these notions, we have the following definition.
\begin{definition}[Discrete Gaussian free field]
The discrete Gaussian free field (DGFF) $\varphi$ on $G$ with zero-boundary conditions on $C$ is the mean-zero multivariate Gaussian indexed on $V$ with density (with respect to the product Lebesgue measure on $\R^{V\setminus C}$) proportional to
\[
\exp\left(-\f12\sum_{i,\,j\in V\setminus C}\f1{2d}\varphi_i(-\Delta(i,\,j))\varphi_j\right).
\]  
For simplicity of notation we now denote the height-one field by $h$. In~\cite{cipriani2022properties} one considers under a suitable rescaling procedure a graph $G_\eps=(V_\eps,\,E_\eps)\subseteq\Z^2$, where both the height-one $h_\eps$ and the DGFF $\vr_\eps$ are defined. One studies then the $\ell$-joint cumulants of first order $\kappa$ of $h$. What one finds is that, when $x^{(1)}_\eps,\,\ldots,\,x^{(\ell)}_\eps$ is a set of $\ell\ge 2$ pairwise distinct points, in the limit $h_\eps$ satisfies
\begin{equation}\label{eq:equal_cumulants}
   \lim_{\eps\to 0} \eps^{-2 \ell}\kappa\left(h_\eps\big(x_{\eps}^{(1)} \big), \dots, h_\eps\big(x_{\eps}^{(\ell)}\big)\right)=-4\ \lim_{\eps\to 0} \eps^{-2 \ell}\kappa\left(C\Phi_\eps\big(x^{(1)}_\eps \big), \dots, C\Phi_\eps\big(x^{(\ell)}_\eps\big)\right) 
\end{equation}
with $C$ a universal explicit constant, and
\[
\Phi_\eps(v)\coloneqq\sum_{i=1}^{2}\left(\vr_\eps(v+e_i)-\vr_\eps(v)\right)^2.
\]
The field $\Phi_\eps$ can be though of as the {\em gradient squared} of the DGFF. One of the goals of these notes is to elucidate the relation between the ``squared norm'' of the DGFF, and the height-one field.
\end{definition}

\subsection{Outline} These notes are organized as follows. Section~\ref{sec:not} recalls some basic probability facts and sets the notation. In Section~\ref{sec:gras} we give the basics of Grassmann calculus, which will be used in Section~\ref{sec:Gauss} to derive properties on Grassmannian ``Gaussians'' and in Section~\ref{sec:ust} on the uniform spanning tree. We conclude with Section~\ref{sec:susy} where we will explain briefly the concept of supersymmetry. 

\section{Mathematical preliminaries}\label{sec:not}
\subsection{Notation} We denote $\N_0\coloneqq\{1,\,2,\,\ldots\}$. We write $[n]\coloneqq\{1,\,2,\,\ldots,\,n\}$ for $n\in \N_0$. We will use boldfonts to denote vectors (for example, $\boxx=(x_1,\,\ldots,\,x_n)$). The set of $n\times n$ square matrices with entries in a field $\F$ is called $M_n(\F)$. The cardinality of a set $A$ is denoted as $\#A$. Given $a\in\C$, we denote by $\Re(a)$ and $\Im(a)$ the real and imaginary part of $a$, respectively.
\medskip
\paragraph{\underline {Matrices}}
\begin{definition}[Positive-(semi)definite matrix] A matrix $A\in M_n(\C)$ is positive-(semi)de\-fi\-ni\-te if \[{\Re}(\overline{\boxx}^T A\boxx)>0\](resp. $\geq0$) for any $\boxx\in\C^n\setminus\{\bf{0}\}$, where $\overline{\boxx}^T$ denotes the transpose complex conjugate of $\boxx$. A similar definition applies to $A\in M_n(\R)$ by replacing $\overline{\boxx}^T$ with $\boxx^{T}$, that is, the transpose of $\boxx$.
    \begin{definition}[Hermitian matrix] A matrix $A\in M_n(\C)$ is Hermitian if $A=\overline{A}^T$, that is, for all $1\leq i,\,j\leq n$ one has $\displaystyle A(i,\,j)=\overline{A(j,\,i)}.$
        
    \end{definition}
\end{definition}
Let $n\in\N$ and $\mathbf X=(X_{i})_{i=1}^n$ be a vector of real-valued random variables, each of which has all finite moments.

\medskip
\paragraph{\underline{Cumulants}}
For a reference on this paragraph see for example~\cite{peccati2011wiener}.
\begin{definition}[Joint cumulants of random vectors]\label{def:cum_mult}
The cumulant generating function $K(\mathbf t)$ of $\mathbf X$ for $\mathbf t=(t_1,\,\dots,\,t_n)\in \R^n$ is defined as 
\[
K(\mathbf t) \coloneqq \log \left ( \mathbb{E}\big[\exp\left( \mathbf{t} \cdot \mathbf{X}\right)\big]\right ) = \sum_{\mathbf m\in\N^n} \kappa_{\mathbf m}(\mathbf X) \prod_{j=1}^n \frac{t_j^{m_j}}{m_j!} \ ,
\]
where $\mathbf t\cdot \mathbf X$ denotes the scalar product in $\R^n$, $\mathbf m=(m_1,\,\dots,\,m_n)\in\N^n$ is a multi-index with $n$ components, and
\[
\kappa_{\mathbf m}(\mathbf X)=\frac{\partial^{|m|}}{\partial t_1^{m_1}\cdots \partial t_n^{m_n}}K(\mathbf t)\Big|_{t_1=\ldots=t_n=0} \ ,
\]
being $|m|=m_1+\cdots+m_n$. 
\end{definition}
The joint cumulant of the components of $\mathbf X$ can be defined as a Taylor coefficient of $K(\mathbf t)$ for $\mathbf m=(1,\,\ldots,\,1)$; in other words
\[
\kappa(\mathbf X)=\frac{\partial^n}{\partial t_1\cdots \partial t_n} K(\mathbf t)\Big|_{t_1=\ldots=t_n=0} \ .
\]
In particular, for any $A\subseteq [n]$, the joint cumulant $\kappa(X_i:\, i\in A)$ of $\mathbf X$ can be computed as 
\[
\kappa(X_i:\, i\in A) = \sum_{\pi \in \Pi(A)} (|\pi|-1)! (-1)^{|\pi|-1} \prod_{B\in \pi} \mathbb{E} \left[\prod_{i\in B} X_i \right] \ ,
\]
with $\Pi(A)$ the set of partitions of the set $A$ and $|\pi|$ the cardinality of $\pi$.
Let us remark that, by some straightforward combinatorics, it follows from the previous definition that
\begin{equation}\label{eq:mom_to_cum}
\mathbb{E} \left[\prod_{i\in A} X_i \right] = \sum_{\pi \in \Pi(A)} \prod_{B\in \pi} \kappa(X_i:\, i\in B) \ .
\end{equation}
If $A=\{i,\,j\}$, $i,\,j\in[n]$, then the joint cumulant $\kappa(X_i,\,X_j)$ is the covariance between $X_{i}$ and $X_{j}$. In addition, for a real-valued random variable $X$, one has the equality
\[
\kappa(\underbrace{X,\,\ldots,\,X}_{n\text{ times}})=\kappa_n(X),\quad n\in \N,
\]
which we call the \emph{$n$-th cumulant of $X$}.
\section{Grassmannian calculus}\label{sec:gras}
For this Section we are indebted to~\cite{abdesselam}.

Let $\F$ be a field which contains the field $\Q$ of rational numbers.
\begin{definition}[Associative algebra]
An associative algebra $\cA$ over a field $\F$ is a vector space over $\F$ with an operation $\wedge$ satisfying the following conditions: for all $x,\,y,\,z\in \cA,\,a\in \F$ one has
\begin{description}
    \item[Left distributivity] $x\wedge(y+z)=x\wedge y+x\wedge z.$
    \item[Right distributivity] $(x+y)\wedge z=x\wedge y+y\wedge z. $
    \item[Compatibility] $a(x\wedge y)=(ax)\wedge y=x\wedge(ay). $
    \item[Associativity] $(x\wedge y)\wedge z=x\wedge(y\wedge z).$
\end{description}
\end{definition}
\begin{example}
    An example of an associative algebra $(\cA,\,\wedge)$ is given by the $n\times n$ matrices $M_n(\R)$ or $M_n(\C)$ with the usual matrix multiplication.
\end{example}
For the rest of the notes, we will omit the symbol $\wedge$ unless there is a risk of confusion with other operations.
In our applications, $\F$ will typically be the set of real numbers $\R$ or the complex numbers $\C$, and we will not work with any other field of numbers.
\begin{definition}[Generators of $\cA$]
    The set $\{\xi_i:\,i\in I\}$ forms a set of generators for $\cA$ if, for all $x\in\cA$, there exists $n\in \N$ such that
    \[
    x=\mathrm{pol}_\F(\xi_{i_1},\,\,\ldots,\,\xi_{i_n}),\quad i_j\in I,\,1\leq j\leq n.
    \]
    In other words every element of the algebra can be written as a finite polynomial in the $\xi's$.
\end{definition}
We are now ready to define a Grassmann algebra.
\begin{definition}
    A Grassmann algebra $\GA$ is an associative algebra over $\R$ (or $\C$) generated by a set of generators $\{\xi_i:\,i\in I\}$ that satisfy the following anticommutation relations:
    \begin{equation}\label{eq:AC}
        \xi_i\xi_j=-\xi_j\xi_i.
    \end{equation}
\end{definition}
\begin{corollary}[Nilpotency of the $\xi$'s]\label{cor:nil}
Since the base field of the algebra contains $\Q$, and in particular contains $\nicefrac{1}{2}$, it follows that
\[
\xi_i^2=0,\,\quad i\in I.
\]
\end{corollary}
\begin{example}
    An example of a Grassmannian algebra is the algebra of differential forms on an $n$-dimensional manifold with the operation $\wedge$ defined as the standard wedge product. To show this, let $V$ be a $n$-dimensional manifold with coordinates $x_1,\, \ldots ,\,x_n$. A differential form on $V$ can then be written as
    \[
    F = F_0 +\cdots+ F_n,
    \]
    where $F_0\in C^{\infty}(V)$ is a $0$-form, i.e., an ordinary function from $V$ to $\R$, and $F_i$ is an $i$-form for any $i\in[n]$, i.e., 
    \[
    F_i(x) = f(x) \De x_{i_1} \wedge \cdots \wedge \De x_{i_n},
    \]
    with $x\in V$ and $f\in C^{\infty}(V)$. The form $F_p$ is the degree-$p$ part of $F$ and a form $F$ has degree $p$ if $F=F_p$. The set $\{\De x_i:\,i=1,...,n\}$ forms a set of generators for this algebra. To conclude the example and show that this is indeed a Grassmann algebra, note that $\De x_i \wedge \De x_j = - \De x_j \wedge \De x_i$ for any $i,j\in[n]$.
\end{example}
Corollary~\ref{cor:nil} entails an important property: all elements of a Grassmann algebra are affine polynomials in the $\xi$'s, as we shall see now.
\begin{definition}[Even and odd, parity]\label{def:evenodd}
A non-zero monomial $f\in \GA$ is called even if it contains an even number of generators, and it is called odd otherwise. We define the parity of such a monomial as
\[
p(f)\coloneqq\begin{cases}
    0 & \text{if }f\text{ is even},\\
    1 & \text{if }f\text{ is odd}.
\end{cases}
\]
We can extend the definition to even resp. polynomials whose monomials are all even resp. odd.
\end{definition}
\begin{lemma}[Graded commutation relations]\label{lem:swapping}
For all non-zero monomials $f,\,g\in \GA$ one has
\[
fg=(-1)^{p(f)p(g)}gf.
\]
\end{lemma}
\begin{proof}
To swap the order of the terms $fg$, one has to bring each generator appearing in $g$ to the front of $f$, yielding a $(-1)^{p(g)}$ for each generator. In addition, each generator in $g$ has to perform a number of swaps equal to the number of factors of $f$, gaining an additional $(-1)^{p(f)}$.
\end{proof}
From Lemma~\ref{lem:swapping} we readily obtain two important consequences. The first is that even polynomials commute, while odd ones anticommute. This gives rise to the nomenclature ``bosons'' for even terms in the algebra and ``fermions'' for the odd ones. Secondly, it allows us to prove immediately Pauli exclusion principle, which we phrase as follows.
\begin{proposition}[Pauli exclusion principle]
If $f$ is an odd element of $\GA$ then $f^2=0$.
\end{proposition}
Even though there are several terms whose square is zero (indeed, one can show that all polynomials whose monomials share at least one generator enjoy this property) Pauli exclusion principle becomes significant, from a physical viewpoint, when $f$ is homogenous of degree 1.
\begin{example}
Not all polynomials squared are zero. Indeed,
\[
(\xi_1\xi_2+\xi_3\xi_4)^2 = 2\prod_{i=1}^4 \xi_1 \neq 0,
\]
where we have used $(\xi_1\xi_2)^2=(\xi_3\xi_4)^2=0$.
\end{example}
\subsection{Functions of Grassmann variables}
\begin{definition}\label{def:func-of-grasmmann}
	Let $f:\, \R \to \R$ be an analytic function given by $f(x)= \sum_{k=0}^\infty a_k x^k $, and let $F \in \GA$ be an element of the Grassmann algebra.
	We define the composition of an analytic function with an element of the Grassmann algebra, $f(F)$, by
	\begin{equation*}
		f(F)\coloneqq \sum_{k=0}^{\infty}
		a_k F^{k}.
	\end{equation*}
\end{definition}
As a consequence of this definition, and the nilpotency of the generators, one can observe for example that for any generator $\xi$
\[
\exp(\xi)=\sum_{k=0}^\infty \f{\xi^k}{k!}=1+\xi.
\]
\begin{proposition}\label{pro:comm-exponential}
If $f,g\in \GA$ are even polynomials, then
\begin{equation}\label{eq:comm-exponential}
\exp(f+g) = \exp(f) \exp(g).
\end{equation}
\begin{proof}
The proof follows because $f$ and $g$ are even polynomials, so their monomials commute like ordinary numbers, where the relation \eqref{eq:comm-exponential} holds. 
\end{proof}
\begin{remark}
    Note that Proposition \ref{pro:comm-exponential} does not necessarily hold for odd polynomials. As an example, since $(\xi_1+\xi_2)^2 = \xi_1\xi_2 + \xi_2\xi_1=0$, we have
    \[
    \exp(\xi_1+\xi_2) = 1 + \xi_1 + \xi_2,
    \]
     whereas
    \[
    \exp(\xi_1) \exp(\xi_2) = 1 + \xi_1 +\xi_2+\xi_1\xi_2.
    \]
\end{remark}
\end{proposition}
\subsection{Differentiation and integration}\label{subsec:int_fer}
From now on, we will simplify the structure of the Grassmann algebra even more by giving it only a finite number of generators $\{\xi_i:\,1\leq i\leq n\}.$
\begin{definition}[Right derivative]
  The right derivative is the linear map $\GA\to \GA$ which acts on monomials $\xi_{i_1}\cdots\xi_{i_p},\,1\leq i_1< \ldots<i_p\leq n$, as
\begin{equation}\label{eq:def_der}
     \begin{cases}
   \f{\partial}{\partial_{\xi_j}} \xi_{i_1}\cdots\xi_{i_p} =
    (-1)^{\alpha-1} \xi_{i_1} \cdots \xi_{i_{\alpha-1}} \xi_{i_{\alpha+1}} \cdots \xi_{i_p} & \text{if there is }1\le \alpha\le p \text{ such that } i_{\alpha} = j, \\
    \hfil 0 & \text{otherwise. }  
  \end{cases}
\end{equation}
\end{definition}
In other words, if $x=x_1+\xi_i x_2$ where $x_1,\,x_2\in \GA$ are Grassmann terms in which $\xi_i$ does not occur, one has
    \[
    \f{\partial}{\partial {\xi_i}}x= x_2.
    \]
    It follows that the derivative is also nilpotent, in the sense that
    \[
    \left(\f{\partial}{\partial \xi_i}\right)^2=0
    \]
    for all $i$. Moreover, the derivative satisfies the {\it Leibniz rule}
    \begin{equation}\label{def:Leibnizrule}
        \f{\partial}{\partial \xi_i} (fg) = \left( \f{\partial}{\partial \xi_i} f \right) g + (Pf) \left(\f{\partial}{\partial \xi_i} g \right).
    \end{equation}
      where $f,g\in \GA$ and $P$ is the {\it parity operator}~\cite[Section 2.2]{wegner2016supermathematics} defined by
    \[
    P \left( \xi \right) = -\xi.
    \]
    This operator essentially multiplies a monomial of $k$ variables by $(-1)^k$ (whence Lemma $\ref{lem:swapping}$ can be rephrased as $fg=P(gf)$).
\begin{definition}\label{def:trasl-of-grasmmann}
Let $f:\, \R \to \R$ be an analytic function and let $F,G\in\GA$ such that $F$ is generated by $\xi_1,...,\xi_n$ and $G$ is an odd polynomial. Then
\begin{equation}\label{eq:trasl}
f(F+G) = f(F) + \sum_{i=1}^n \dfrac{\partial}{\partial \xi_i} f(F) G.
\end{equation}
\end{definition}
\begin{remark}
    Note that the oddness of $G$ is crucial for \eqref{eq:trasl} to make sense. Indeed, in that case this substitution can be expressed in terms of a {\it finite} Taylor series.
\end{remark}
    \begin{definition}[Grassmann--Berezin integration]
The Grassmann--Berezin integral is defined as
\[
\int \D(\boxi) x  =\int \De \xi_1\De \xi_2\cdots\De\xi_n x\coloneqq\partial_{\xi_{1}}\partial_{\xi_{2}} \cdots \partial_{\xi_n}F, \quad x\in \GA.
\]
    \end{definition}
    Surprisingly the integral is defined in the same way as the derivative (actually, the two coincide!). Even though surprising, this definition retains many useful properties of the integral, for example linearity, or the fact that an integral does not depend on the variables which are integrated out. It is however crucial to keep track of the order of integration because of~\eqref{eq:def_der}. For example
    \[
    \int \De \xi_1\De\xi_2( \xi_2\xi_1)=1
    \]
    but
    \begin{eqnarray*}
    \int \De \xi_2\De\xi_1(\xi_2\xi_1)=\f{\partial}{\partial\xi_2}\left(\f{\partial}{\partial\xi_1}(-\xi_1\xi_2)\right)=-\f{\partial}{\partial\xi_2}\xi_2=-1.
    \end{eqnarray*} For this reason, to fix notation we will now start working with a special Grassmann algebra: the one in which $n=2m$ and the set of generators is divided into two sets of variables $\{\xi_i,\,\bxi_i:\,1\leq i\leq m\}$. The bar is reminiscent of complex conjugation, but we consider it only as a special notation to distinguish the two sets of variables. In this setup, we choose
    \begin{equation}\label{eq:BerLeb}
    \int \Diff x\coloneqq\int \left(\prod_{i=1}^n
\De\xi_i\De\bxi_i\right) x ,\quad x\in \GA.    \end{equation}
It is important to observe that integration satisfies
\begin{equation}\label{eq:int_rule}
    \int\De \xi\De \bxi \left(a_1+a_2\xi+a_3\bxi+a_4\bxi\xi\right)=a_4,\quad a_i\in \R,\,1\leq i\leq 4
\end{equation}
where $\xi,\,\bxi$ are arbitrary generators. In other words, only polynomial expressions in which all generator appear give a non-zero contribution to integrals.

There are many properties of the Grassmann integral that work exactly in the same way as for standard integrals: invariance under translation, change-of-variables formulas and Fubini's theorem. We will prove them now for completeness, beginning with translation invariance. The fact that it holds is intuitively clear because the integral is ``morally'' a derivative, and thus does not see shifts. Let's see the proof more precisely.
\begin{proposition}[Invariance under translation for Grassmann--Berezin integration]
Let $I=\{i_1,...,i_p\}$ be an ordered sequence of indices in $[n]$ and let $\chi_1,...,\chi_n$ be odd elements of $\GA$ satisfying $\chi_j=0$ for any $j\notin I$. Then
\begin{equation}\label{eq:invtrasl}
    \int \De \xi_{i_p} \cdots \De \xi_{i_1} f(\xi+\chi) = \int \De \xi_{i_p} \cdots \De \xi_{i_1} f(\xi),
\end{equation}
where $f(\xi+\chi)$ denotes the substitution defined in \eqref{eq:trasl}.
\end{proposition}
\begin{proof}
    To prove the statement, it is enough to show that \eqref{eq:invtrasl} can be written as 
    \[
    \f{\partial}{\partial \xi_{i_p}} \cdots \f{\partial}{\partial \xi_{i_1}} f(\xi+\chi) = \f{\partial}{\partial \xi_{i_p}} \cdots \f{\partial}{\partial \xi_{i_1}} f(\xi).
    \]
    To prove this, it suffices to consider the cases in which $f(\xi)=\xi_{i_1}\cdots\xi_{i_p}$ for some sequence of indices $\{j_1,...,j_m\}\subseteq[n]$. For any $i\in I$ and $j\in[n]$ we have that
    \[
    \f{\partial}{\partial \xi_i} (\xi_j+\chi_j) = \f{\partial}{\partial \xi_i} \xi_j = \delta_{ij}.
    \]
    Using this relation together with the Leibniz rule, we note that $\f{\partial}{\partial \xi_i} (\xi_{j_1}+\chi_{j_1}) \cdots (\xi_{j_m}+\chi_{j_m})$ equals the same object in which $\chi_i$ has been replaced by zero. By iterating this for $\f{\partial}{\partial \xi_{i_1}}\cdots\f{\partial}{\partial \xi_{i_p}}$, we set $\chi_i=0$ for any $i\in I$. However, by hypothesis these are the only nonzero $\chi_i$. This concludes the proof.
\end{proof}
Changing variables in Grassmann--Berezin integration is, unlike translation invariance, a rule that defies its standard counterpart. The rule is as follows.
\begin{proposition}[Linear change of variables for Grassmann--Berezin integration]
    Let $A\in M_n(\R)$ and let $f:\mathbb{R}\to\mathbb{R}$ be an analytic function. Define new Grassmannian variables by $\chi_i=\sum_{j=1}^n A_{ij}\xi_j$, $i=1,...,n$. Then
    \begin{equation}\label{eq:change_var}
    \int \De \xi_n \cdots \De \xi_1 f(\xi_1,...,\xi_n) = \det(A) \int \De \chi_n \cdots \De \chi_1 f(\chi_1,...,\chi_n).
    \end{equation}
\end{proposition}
\begin{proof}
Let $\mathfrak S_n$ denote the set of permutations of $n$ elements and $\hbox{sgn}(\sigma)$ the sign of the permutation $\sigma$. The statement follows after noting that
\[
\begin{array}{ll}
\displaystyle \int \De \xi_n \cdots \De \xi_1 \chi_1 \cdots \chi_n &= \displaystyle \int \De \xi_n \cdots \De \xi_1 \sum_{j_1=1}^n A_{ij_1}\xi_{j_1} \cdots \sum_{j_n=1}^n A_{ij_n}\xi_{j_n} \\
&= \displaystyle \sum_{\sigma\in\mathfrak S_n} \hbox{sgn}(\sigma) \prod_{i=1}^n A_{i\sigma(i)} \\
&= \det(A).
\end{array}
\]
\end{proof}

Note that for ordinary multivariate integrals the factor $\det(A)$ should be on the other side compared to~\eqref{eq:change_var}: if $\mathbf u=A\mathbf x$ then
\[
\int_{\R^n} f(\mathbf x) \De x_1\ldots\,\De x_n=\int_{\R^n}f\left(A^{-1}\mathbf u\right)\frac{1}{|\det A|}\De\mathbf u.
\]
Finally we can state the analog of Fubini's theorem, which essentially mimics its standard counterpart up to a possible sign change.
\begin{proposition}[Fubini theorem for Grassmann--Berezin integration]
Let $I=\{i_1,...,i_p\}$ be an ordered sequence of indices in $[n]$, and let $I^c=\{j_1,...,j_{n-p}\}$. Then, for any elements $f\in \GA$ generated by $\xi_{i_1},...,\xi_{i_p}$ and $g\in \GA$ generated by $\xi_{j_1},...,\xi_{j_{n-p}}$, we have
\[
\int \De \xi_{i_p} \cdots \De \xi_{i_1} \De \xi_{j_{n-p}} \cdots \De \xi_{j_1} fg = (-1)^{p(n-p)} \left( \int \De \xi_{i_p} \cdots \De \xi_{i_1} f \right) \left( \int \De \xi_{j_{n-p}} \cdots \De \xi_{j_1} g \right).
\]
\end{proposition}
\begin{proof}
    Expanding $f$ and $g$ in monomials, we see that the unique terms that contribute to the integrals are $\xi_{i_1}\cdots\xi_{i_p}$ and $\xi_{j_1}\cdots\xi_{j_{n-p}}$ in $f$ and $g$, respectively. We can then assume without loss of generality that $f=\xi_{i_1}\cdots\xi_{i_p}$ and $g=\xi_{j_1}\cdots\xi_{j_{n-p}}$. By using the derivative rule \eqref{eq:def_der} and the Leibniz rule \eqref{def:Leibnizrule}, we get
    \[
    \int \De \xi_{j_{n-p}} \cdots \De \xi_{j_1} fg = (-1)^{p(n-p)} f \left( \int \De \xi_{j_{n-p}} \cdots \De \xi_{j_1} g \right).
    \]
    The result then follow after integrating both sides with respect to $\De \xi_{i_p} \cdots \De \xi_{i_1}$.
\end{proof}
\section{Grassmannian calculus and Gaussian random variables}\label{sec:Gauss}
This Subsection is devoted to studying ``Gaussian integrals'' in the Grassmannian setting. In particular, let $A\in M_n(\R)$. By using the notation $(\boxbxi,\,A\boxi)\coloneqq \sum_{i,\,j=1}^n \bxi_i A(i,\,j)\xi_j$ we wish to compute the analog of a Gaussian measure, that is, objects of the form
\[
\int \Diff\e^{(\boxbxi,\,A\boxi)}.
\]
We will prove that
\begin{proposition}\label{pro:gau_int}
    \[
\int \Diff\e^{(\boxbxi,\,A\boxi)}=\det(A).
\]
\end{proposition}
Before we enter into the proof, let us note that there are no assumptions needed on $A$, unlike the real case in which
\[
\int_{\R^n} \prod_{i=1}^n\f{\De x_i}{\sqrt{2\pi}}\,\e^{-\f12(\boxx,\,A\boxx)}=\left(\det A\right)^{-\f{n}2}
\]
requires $A$ to be symmetric and positive-definite.

\begin{proof}[Proof of Proposition~\ref{pro:gau_int}]
    Since the polynomial $(\boxbxi,\,A\boxi)$ is even, by Proposition \ref{pro:comm-exponential}
    \begin{equation}\label{eq:prodsplit}
        \exp{(\boxbxi,\,A\boxi)} = \prod_{i=1}^n \exp{\left(\bxi_i\sum_{j=1}^n A_{ij}\xi_j\right)}
        = \prod_{i=1}^n \left( 1+ \bxi_i\sum_{j=1}^n A_{ij}\xi_j\right).
    \end{equation}
    When expanding \eqref{eq:prodsplit}, since we are integrating with respect to $\Diff$, only the term containing $\bxi_i$ for all $i=1,...,n$ survives integration, which gives rise to
    \[
     \prod_{i=1}^n \bxi_i \left( \sum_{j=1}^n A_{ij}\xi_j\right) = \prod_{i=1}^n \bxi_i \left( \sum_{j_i=1}^n A_{i{j_i}}\xi_{j_i}\right).
    \]
    Similarly, the terms that give non-vanishing contribution after integrating must contain $\xi_1\cdots\xi_n$ up to permutation. They have the form
    \begin{equation}\label{eq:permsum}
    \sum_{\sigma\in\mathfrak S_n} A_{1j_{\sigma(1)}} \cdots A_{nj_{\sigma(n)}} \bxi_n \xi_{j_{\sigma(n)}} \cdots \bxi_1 \xi_{j_{\sigma(1)}}.
    \end{equation}
    Putting $\bxi_n \xi_{j_{\sigma(n)}} \cdots \bxi_1 \xi_{j_{\sigma(1)}}$ into a standard order, i.e., $\bxi_n \xi_n \cdots \bxi_1\xi_1$, yields the sign of the permutation $\sigma$, so that \eqref{eq:permsum} becomes
    \[
    \sum_{\sigma\in\mathfrak S_n} \prod_{i=1}^n A_{ij_{\sigma(i)}} \hbox{sgn}(\sigma) \bxi_n \xi_n \cdots \bxi_1\xi_1.
    \]
    Since 
    \[
    \int \Diff \bxi_n \xi_n \cdots \bxi_1\xi_1 =1,
    \]
    we get the claim.
\end{proof}
Next we state Wick's theorem which, like in the real case, allows one to compute multipoint functions for a Gaussian vector. In the statement, we are thinking of $\boxi$ resp. $\boxbxi$ as $(1\times n)$-dimensional vectors. The version of the Theorem we need is the following, but more can be found in~\cite[Theorem A.16]{sportiello}, together with their proofs. We will use, for a matrix $A\in M_n(\R)$, the notation
    \[
    A_{i^\c,\,j^\c},\quad i,\,j\in [n]
    \]
for the submatrix obtained removing row $i$ and column $j$ from $A$.
\begin{theorem}[Wick’s theorem]\label{thm:sportiello}
Let $A $ be an ${n\times n}$, $B$ an $ {r\times n}$ and $C$ an $ {n\times r}$ matrix respectively with coefficients in $\R$.
For any sequences of indices $I = 
\{i_1,\,\dots,\,i_r\}$ and $J = \{j_1,\,\dots,\,j_r\}$ in $[n]$ of the same length $r$, if the matrix $A$ is invertible we have
\begin{enumerate}[label=\arabic*.,ref=\arabic*.]
    \item\label{thm_Wick_one} For all $i\in I$, $j\in J$ one has $\int\Diff   \bxi_{i}\xi_{j}\exp\left((\boxbxi, A\boxi)\right) = \det(A) \det\left(A_{i^\c,\,j^\c}\right)$.
    \item\label{thm_Wick_two} $\int\Diff \prod_{\alpha=1}^r \left(\boxbxi^T C\right)_\alpha (B\boxi)_\alpha\exp\left((\boxbxi, A\boxi)\right) = \det(A) \det\left(BA^{-1}C\right)$.
\end{enumerate}
If $\#I \neq \#J$, the integral is $0$.
\end{theorem}

\section{Grassmannian calculus and uniform spanning trees}\label{sec:ust}
At this stage we would like to show one main area of application for Grassmannian variables: studying spanning trees. Indeed we will show that Grassmann variables can completely describe the edge probabilities of so-called uniform spanning tree, whose definition we will recall now.

For the rest of this Section, we will let $G=(V,\,E, \,w)$ be a finite connected graph with vertex set $V$ and edge set $E$. We will assume that each edge $e=(u,\,v)$, $u,\,v\in V$, has an edge weight $w_e=w_{uv}>0$. In particular, edges are unoriented.
\begin{definition}[Spanning tree]
 A finite subgraph $T\subseteq G$ is called a spanning tree of $G$ if 
 \begin{itemize}
     \item it has no cycles, i. e. there is no non-empty subset of the edge set of $H$ that forms a path\footnote{A graph path is a sequence $\{v_1,\,v_2,\,\ldots,\,v_k\}$ of distinct vertices such that $(x_i,\,x_{i+1})$ are graph edges for all $1\leq i\leq k-1$.} such that the first node of the path corresponds to the last;
     \item it is connected, and
     \item it is spanning, i. e. every $v\in V$ has at least one edge of $H$ incident to it.
 \end{itemize}  
\end{definition}
It is easy to see that any connected and finite graph possesses a finite number $T_G$ of spanning trees (and that this number is non-zero). Therefore it is legit to define a uniform probability on the set of spanning trees.
\begin{definition}[UST]
   The uniform spanning tree is the probability measure on the set of all spanning trees of $G$ with probability mass function
   \[
   \prob(t)=\f1{T_G}
   \]
   for $t$ a spanning tree of $G$.
\end{definition}
In order to understand the UST measure one has to get hold of the constant $T_G$, more specifically needs to count the number of spanning trees. To this end, we will need the following object: the Laplacian matrix, which we have briefly encountered on the square lattice in the Introduction.
\begin{definition}[Laplacian matrix]
    The Laplacian (matrix) $\Delta=\Delta(G)$ is the matrix indexed over $V$ defined as
    \[
    \Delta(u,\,v)\coloneqq
    \begin{cases}
     w_{uv}&\text{if }e=(u,\,v)\in E\\
     -\deg_G(u)\coloneqq-\sum_{v\in V}w_{uv} &\text{if }u=v\\
     0 & \text{otherwise}
    \end{cases}.
    \]
\end{definition}
It follows from the definition that $0$ is an eigenvalue for $\Delta$ with eigenvector $$\underbrace{(1,\,1,\,\ldots,\,1)}_{|V|\text{ times}}.$$
One can in fact prove~\cite[Chapter 1]{Chung} that all eigenvalues of $-\Delta$ are real and non-negative, and that the eigenvalue zero has multiplicity one. We are now ready to prove Kirchhoff's theorem, or the matrix-tree theorem~\cite{Kirchhoff}. This Theorem allows us to count the number of spanning trees as a determinant of (a submatrix of) the Laplacian. We will give a ``Grassmannian proof'' for it, which is due to~\cite{degrandi}.
\begin{theorem}[Matrix-tree theorem]\label{thm:mattree}
    Choose an arbitrary $o \in V$. Then
    \[
    T_G=\det(-\Delta_{o^\c,\,o^\c}).
    \]
\end{theorem}
\begin{proof}
Let us call $-\Delta_{o^\c,\,o^\c}\coloneqq O$. By Theorem~\ref{thm:sportiello}~\ref{thm_Wick_one} we have that
\begin{equation}\label{eq:intO}
\int \Diff \bxi_{o}\xi_o\e^{\left(\boxbxi,\,O\boxi\right)}=\det(O).
\end{equation}
Therefore we need to show that the left-hand side above counts the number of spanning trees of $G$. We expand the exponential as follows: fixing $u$, and taking into account the nilpotency of the generators, we have
\begin{eqnarray*}
   \exp\left(\sum_{v=1}^{\#V}\bxi_u O(u,\,v)\xi_v\right) =1+\sum_{v=1}^{\#V} \bxi_u O(u,\,v)\xi_v=1-\sum_{v\neq u}\bxi_u \xi_v w_{uv}+\sum_{v\neq u}\bxi_u \xi_u w_{uv}.
\end{eqnarray*}
By symmetry, a term like $w_{uv}$ can appear together with a pair of variables of the same index ($\bxi_u \xi_u$ or $\bxi_v \xi_v$) or with a pair of variables with different indices ($\bxi_u \xi_v$ or $\bxi_v \xi_u$). We will now count graphically each appearance, encoding it with a different type of arrow (illustrated in Figure~\ref{fig:arrow}).
\begin{figure}
    \centering
    \includegraphics[width=0.5\linewidth]{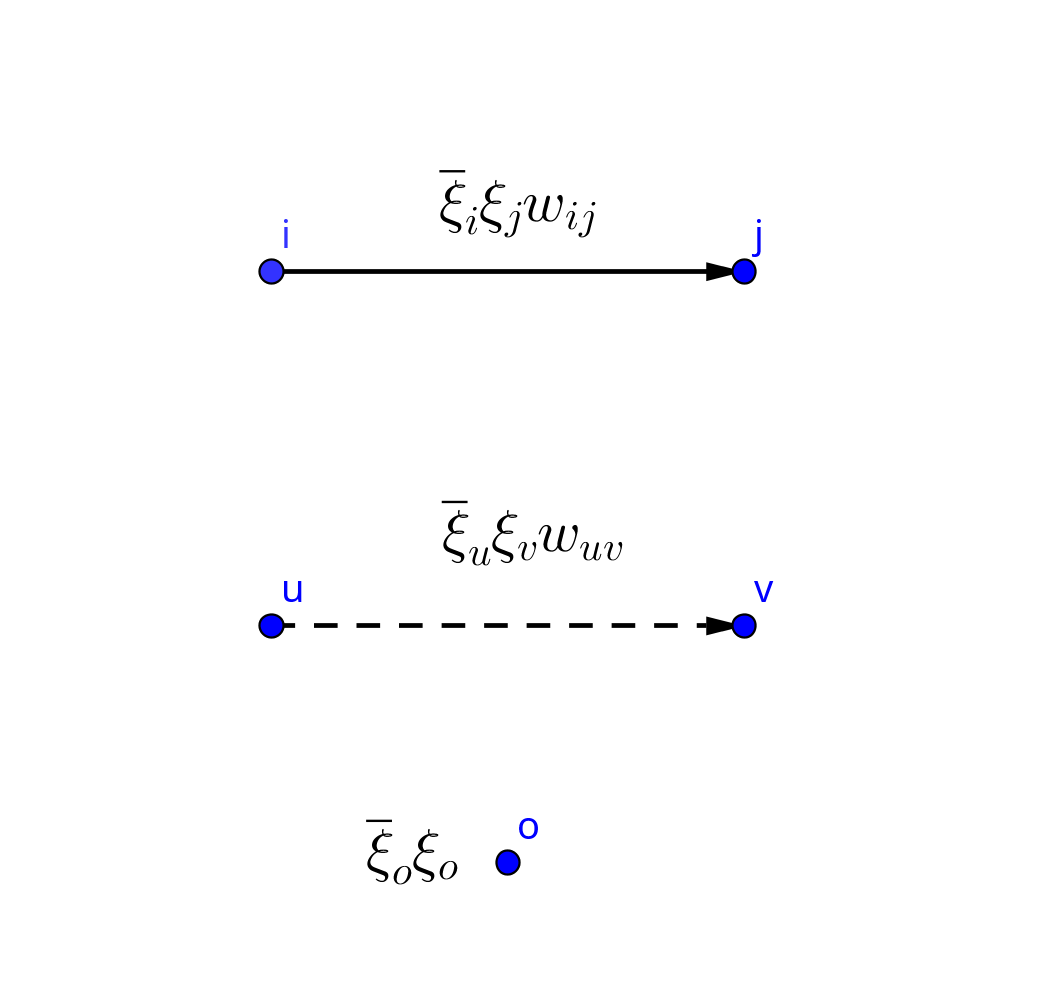}
    \caption{Arrow encoding for the proof of Theorem~\ref{thm:mattree}.}
    \label{fig:arrow}
\end{figure}

If $w_{uv}$ appears in the first instance, we will draw a solid arrow going from $u$ to $v$. If $w_{uv}$ appears in the second instance, we will draw a dashed arrow from $u$ to $v$. For $\bxi_o \xi_o$, which appears in the integrand as well, we will draw no in- or outgoing arrow from $o$. With this encoding, every polynomial in the left-hand side of~\eqref{eq:intO} that has a pair of variables at the same site will have a corresponding outgoing arrow from there, while for a pair of fields with one ``real'' variable and one ``conjugate'' variable the arrow will be dashed (directed from the ``conjugate'' site to the ``real'' one). 

By~\eqref{eq:int_rule} one must have exactly one conjugate variable per site. If a site is visited by the tail of a solid arrow, like node $k$ in Figure~\ref{fig:arrow_three}, or it is $o$, the pair of variables is already complete, so this vertex can only be visited by an arbitrary number of solid arrows, that bring no variable to the head. If a site is visited by the tail of a dashed arrow, like node $j$ in Figure~\ref{fig:arrow_three}, to complete the pair of variables it must also be visited by a dashed arrow-head, and by arbitrarily many solid arrow-heads.
\begin{figure}[ht!]
\definecolor{qqqqff}{rgb}{0.,0.,1.}
\definecolor{ttttff}{rgb}{0.2,0.2,1.}
\definecolor{cqcqcq}{rgb}{0.7529411764705882,0.7529411764705882,0.7529411764705882}
\centering
\begin{tikzpicture}[line cap=round,line join=round,>=triangle 45,x=1.0cm,y=1.0cm]
\clip(-5.32509977827051,-14.2525942350332) rectangle (7.923680709534364,-3.9306430155210466);
\draw [->,line width=1.6pt,dash pattern=on 5pt off 5pt] (-2.,1.) -- (2.,1.);
\draw [->,line width=1.6pt] (0.8602660753880247,-7.559911308203959) -- (-1.0129046563192916,-7.872106430155189);
\draw [->,line width=1.6pt] (-2.417782705099779,-6.252594235033241) -- (-1.0129046563192916,-7.872106430155189);
\draw [->,line width=1.6pt] (8.9236807095343653,-6.155033259423486) -- (5.,-8.);
\draw [->,line width=1.6pt] (-1.0129046563192916,-7.872106430155189) -- (-1.,-9.);
\draw [->,line width=1.6pt,dash pattern=on 5pt off 5pt] (5.,-8.) -- (5.016363636363633,-9.276984478935672);
\draw [->,line width=1.6pt] (3.767583148558755,-11.442838137472236) -- (1.8749002217294877,-12.203813747228331);
\draw [->,line width=1.6pt] (-0.4080266075388041,-10.350155210642972) -- (1.8749002217294874,-12.203813747228331);
\draw [->,line width=1.6pt] (-0.02124168514412259,-13.496496674057598) -- (1.8749002217294877,-12.203813747228331);
\draw [->,line width=1.6pt,dash pattern=on 5pt off 5pt] (3.0066075388026583,-7.442838137472252) -- (5.,-8.);
\begin{scriptsize}
\draw [fill=qqqqff] (0.,-1.) circle (2.5pt);
\draw[color=qqqqff] (-5.227538802660754,-3.716008869179584) node {$o_1$};
\draw [fill=qqqqff] (-1.0129046563192916,-7.872106430155189) circle (2.5pt);
\draw[color=qqqqff] (-0.8763192904656333,-7.520886917960057) node {$k$};
\draw [fill=qqqqff] (5.,-8.) circle (2.5pt);
\draw[color=qqqqff] (5.0368070953404,-7.657472283813715) node {$j$};
\draw [fill=qqqqff] (1.8749002217294877,-12.203813747228331) circle (2.5pt);
\draw[color=qqqqff] (2.011485587583146,-11.852594235033209) node {$o$};
\end{scriptsize}
\end{tikzpicture}
\caption{Skematic illustration of arrow configurations for Theorem~\ref{thm:mattree}.}
\label{fig:arrow_three}
\end{figure}
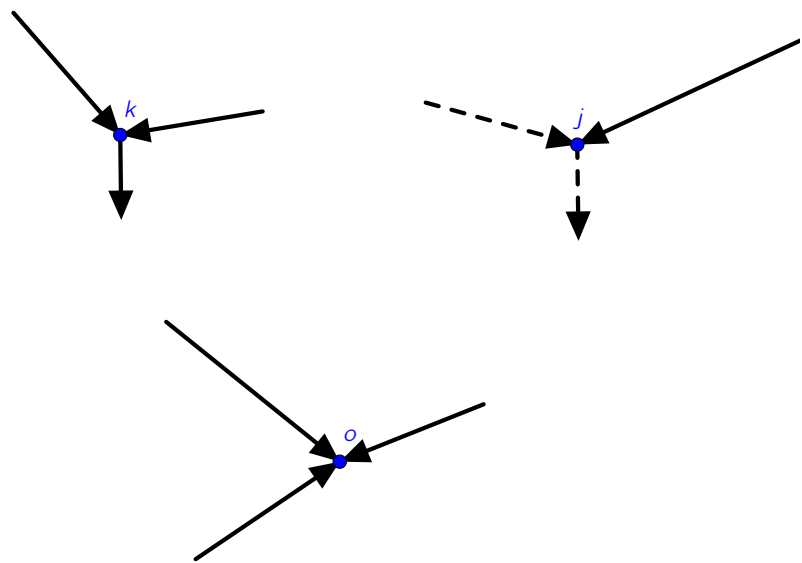
We therefore deduce two statements:
\begin{enumerate}[leftmargin=*]
    \item dashed arrows come in closed, self-avoiding circuits (since for each outgoing arrow there is one and only one incoming arrow);
    \item solid arrows create a subgraph of $G$ whose connected components have the following property. For each connected component, there must be a ``root structure'' such that, for each vertex in the component, either the vertex lies in the root structure or there is a single path which connects it to the root, touching it only at the last vertex. This path is oriented towards the root structure, in other words it is a tree rooted at the ``root structure''.
\end{enumerate}
We need know to understand what kind of root structures we can have. The possibilities are:
\begin{enumerate}[label=(\roman*),ref=(\roman*)]
    \item the vertex $o$;
    \item\label{typeii} a closed oriented circuit of dashed arrows;
    \item\label{typeiii} a closed oriented circuit of solid arrows.
\end{enumerate}
We now show that a polynomial appearing in the integral of~\eqref{eq:intO} giving rise to a root structure of type~\ref{typeii} cancels exactly the integral contribution of a polynomial with a root structure of type~\ref{typeiii}. Indeed, if there exists an oriented cycle of dashed arrows on $\ell$ vertices, which without loss of generality we call $\{1,\,2,\,\ldots,\,\ell\}$, then the associated polynomial must be of the form
\begin{align}
    \left(-\bxi_1\xi_2 w_{12}\right) \left(-\bxi_2\xi_3 w_{23}\right)\cdots \left(-\bxi_\ell\xi_1 w_{1\ell}\right)&\stackrel{\eqref{eq:AC}}{=}-\bxi_1 w_{12} \left(\bxi_2\xi_2 w_{23}\right)\cdots \left(\bxi_\ell\xi_\ell w_{1\ell}\right)\xi_1\nonumber\\
    &=-\left(-\bxi_1\xi_1 w_{12}\right) \left(\bxi_2\xi_2 w_{23}\right)\cdots\left(\bxi_\ell\xi_\ell w_{1\ell}\right)\label{eq:cycle}
\end{align}
where in the last equality we have used that the monomials $\bxi_i\xi_i$ are even. The proof is finished once one notices that the right-hand side of~\eqref{eq:cycle} is of type~\ref{typeiii}.
\end{proof}
Thanks to Proposition \ref{pro:gau_int}, we can define the Gaussian integration on $\GA$ via its moments  as
\[
\left< F \right> \coloneqq (\det(O))^{-1} \int \Diff\e^{(\boxbxi,\,O\boxi)} F
\]
for all $F\in\GA$ and $\{\xi_i,\bxi_i\}_{i\in V}$. In view of Theorem \ref{thm:mattree}, this means that
\[
\left< 1 \right> =1
\]
which intuitively justifies the idea of ``Gaussian probability measure'' on the Grassmann algebra.

In what follows we find a deeper relation between the $\prob$-probability to have some fixed edges in a spanning tree and the expectation $\left< \cdot \right>$ of some fermionic polynomials. We will use throughout this Section the notation $O=-\Delta_{o^\c,\,o^\c}$, and use $O^{-1}$ to denote its inverse, which exists by Theorem~\ref{thm:mattree}. 

Fix an arbitrary orientation of the edges $E$. We will denote an oriented edge as $\vec{xy}$. Even if we have never used oriented edges before, one can prove that this choice does not matter towards the next result (and we will indeed give a ``fermionic'' explanation for this). We let
\begin{equation}\label{eq:transf-current}
T(\vec{xy},\vec{uv}) = O^{-1} (x,\,u) - O^{-1} (y,\,u) - O^{-1} (x,\,v) + O^{-1} (y,\,v),
\end{equation}
be the transfer-impedance matrix (see~\cite[Section 4]{Burton1993Jul}). A classical result of~\cite[Theorem 4.2]{Burton1993Jul} states that the edge probabilities of the uniform spanning tree can be expressed as determinant of the transfer-impedance matrix.
\begin{theorem}[Burton-Pemantle theorem]\label{thm:B-P}
For any finite collection of disjoint undirected edges $e_1,...,e_k\in E$,
\[
\prob(e_1,...,e_k\in t) = \det\left(T(e_i,e_j)_{i,j=1}^k\right)
\]
for $t$ a spanning tree of $G$, where the matrix $T$ is defined in \eqref{eq:transf-current}.
\end{theorem}
Define now, for any edge $e=(uv)\in E$,
\begin{equation}\label{eq:zeta--variables}
\zeta_e \coloneqq w_{uv}(\bxi_u - \bxi_v) (\xi_u - \xi_v).
\end{equation}
Note that the definition of $\zeta_e$ is independent of the orientation of $e$.
\begin{lemma}[{\cite[Lemma 4.5]{chiarini2023fermionicgaussianfreefield}}]\label{lmm:LTG}
If $\emptyset\neq F\subseteq E$, then
\[
\left< \prod_{f\in F} \zeta_f \right> = \det\left(T(f,f')_{f,f'\in F}\right).
\]
\end{lemma}
\begin{proof}
Without loss of generality let us assume that $F=\{f_1,\,\ldots, \,f_k\}$, $k\geq 1$, and each edge $f_i$ is oriented as $\overrightarrow{u_i v_i}$.
 We observe that
\begin{equation*}	
	\left(\bxi_{u_i}-\bxi_{v_i}:\,i=1,\,\ldots,\,k\right) =
	\boxbxi^{T} C,
	\quad
	\left(\bxi_{u_i}-\bxi_{v_i}:\,i=1,\,\ldots,\,k\right) = {B}\boxi,
\end{equation*}
where $ B= C^T$ and $ C$ is a $ (\#V-1)\times k$ matrix such that the column corresponding to the $i$-th point is given by 
\[ 
 C(\cdot,\,i) = (0,\,\dots,\,0,\,-1,\,0,\,\dots,\,0,\,1,\,0,\,\dots,\,0)^T , 
\]
with the $-1$ (resp. $1$) located at the $v_i$-th position (resp. at the $u_{i}$-th position).
Therefore,
\begin{align}   \label{eq:aux_two}
	\left< \prod_{f\in F} \zeta_f \right>=
	\int\Diff
	\left(\prod_{\alpha=1}^k\left(\boxbxi^{T} C\right)_\alpha \left({B}\boxi\right)_\alpha\right)
	\exp\left(\boxbxi,\,O\boxi\right).
\end{align} 
The lemma now follows from item \ref{thm_Wick_two} of Theorem~\ref{thm:sportiello} and the computation
\[
\left( BO^{-1} C \right) (f,\,g) = T(f,\,g)
\]
for $f,\,g \in E$.
\end{proof}
As a comment to this lemma, note that the independence of $\zeta$ from the direction of the edges explains why the determinant of the transfer matrix, as already mentioned, is not influenced by our arbitrary choice of the edge orientation.

The following Theorem is the Grassmannian analogue of~\cite[Corollary 4.4]{Burton1993Jul}, which is stated without proof in~\cite[Equation (3)]{Jan2019Feb} and proven in~\cite[Proposition 3.6]{Rapoport2023Dec}.
\begin{theorem}[Edge-representation with fermions]
For any $\emptyset\neq F,\,F'\subseteq E$ such that $F\cap F'=\emptyset$, 
\[
\prob(F\in t, F'\cap t = \emptyset) = \left< \prod_{f\in F} \zeta_f \prod_{f'\in F'} (1-\zeta_{f'}) \right>.
\]
\end{theorem}
\begin{proof}
    Note that, since
    \[
    \prod_{f'\in F'} (1-\zeta_{f'}) = \sum_{\gamma\subseteq F'} (-1)^{|\gamma|} \prod_{f'\in F'} \zeta_{f'},
    \]
    we can write
    \begin{equation}\label{eq:comp1}
    \left< \prod_{f\in F} \zeta_f \prod_{f'\in F'} (1-\zeta_{f'}) \right> = \sum_{\gamma\subseteq F'} (-1)^{|\gamma|} \left< \prod_{f\in F} \zeta_f \prod_{f'\in F'} \zeta_{f'} \right>.
    \end{equation}
    By Lemma \ref{lmm:LTG} and Theorem \ref{thm:B-P}, we have that \eqref{eq:comp1}, for $t$ a spanning tree, reduces to
    \begin{equation}\label{eq:comp2}
    \sum_{\gamma\subseteq F'} (-1)^{|\gamma|} \prob(F\cup\gamma\in t) = \prob(F\cup\gamma\in t) - \sum_{\#\gamma=1, \atop \gamma\subseteq F'} \prob(F\cup\gamma\in t) + \sum_{\#\gamma=2, \atop \gamma\subseteq F'} \prob(F\cup\gamma\in t) + \dots
    \end{equation}
    Now we consider separately the two following cases:
    \begin{itemize}
        \item[(i)] $\prob(F\subseteq t)=0$;
        \item[(ii)] $\prob(F\subseteq t)\neq0$.
    \end{itemize}

     \medskip
    {\bf Case (i).} Note that in this case \eqref{eq:comp2} equals 0, which proves our claim.
    
     \medskip
    {\bf Case (ii).} We can now divide each term of \eqref{eq:comp2} by $\prob(F\subseteq t)$. Thus,
    \begin{align}
    \dfrac{1}{\prob(F\subseteq t)} \left< \prod_{f\in F} \zeta_f \prod_{f'\in F'} (1-\zeta_{f'}) \right>
    &= 1 - \prob(\exists \ \emptyset\neq\gamma\subseteq F' : \gamma\subseteq t | F\subseteq t) \nonumber \\
    &= \prob\left((\exists \ \emptyset\neq\gamma\subseteq F' : \gamma\subseteq t)^c | F\subseteq t\right) \nonumber \\
    &= \prob(F'\cap t =\emptyset| F\subseteq t),  \label{eq:comp3}
    \end{align}
    where the second equality follows from the inclusion/exclusion principle. To conclude, it suffices to multiply both terms of \eqref{eq:comp3} by $\prob(F\subseteq t)$.
\end{proof}
\subsection{Real and complex Gaussians}
We will now give a brief recap of Gaussian integration in complex variables. We will restrict ourselves to centered vectors but a more general theory can of course be studied, as for example in the reference~\cite{complex} where the results we are presenting are proved. 

Let $A\in M_n(\C)$ be an Hermitian positive-definite matrix with complex entries.
\begin{definition}[Complex Gaussian random vector]
  A random vector ${\boldsymbol{Z}}=(Z_1,\,\ldots,\,Z_n)\in\C^n$ has a complex Gaussian distribution with mean zero and inverse covariance matrix $A$ if it has a density equal to
  \[
  |\det(A)|\exp\left(-\left(\overline{\bf Z},\,A{\bf Z}\right)\right)
  \]
  with respect to the measure
  \[
  \prod_{\alpha=1}^n\f{\De \Re(Z_\alpha)\De \Im(Z_\alpha)}{\pi}.
  \]
  We thus have 
  \[
  A^{-1}=\Exp\left[\bf Z\overline{\bf Z}^T\right],
  \]
  in other words for all $1\leq i,\,\,j\leq n$ it holds that
  \begin{align*}
      &\Exp[Z_i Z_j]=0\\
      &\Exp[Z_i \overline{Z}_j]=A^{-1}(i,\,j)=\overline{A^{-1}(j,\,i)}\\
      &\Exp[|Z_i|^2]=A^{-1}(i,\,i).
  \end{align*}
  Equivalently, for all $1\leq i\leq n$ one has $Z_i=X_i+\iu Y_i$, where $({\bf X},\,{\bf Y})$ is a $2n$-dimensional Gaussian vector with
  \begin{align*}
  \left(\bf X \atop \bf Y \right)\sim \mathcal N\left( \genfrac(){0pt}{0}{\bf 0}{\bf 0},\;\frac12\mat{\Re A^{-1} & -\Im A^{-1} \\ \Im A^{-1} & \Re A^{-1}}\right).
  \end{align*}
\end{definition}
\subsection{Cumulants of Gaussian vectors}
A characterisation of real univariate Gaussian random variable is that all its cumulants of order at least three vanish. More can be said about cumulants of Gaussian vectors, and we will now proceed to proving what the cumulants of the vector of their squares look like. Recall now the notation $\mathfrak S_n$ for the set of permutations of $n$ elements, and $\#\sigma$ to denote the number of cycles in the cyclic decomposition of $\sigma \in \mathfrak S_n.$ Both results are taken from~\cite{mccullagh2006permanental}.
\begin{proposition}[Cumulants of real Gaussian vector]\label{prop:cumreal}
   Let $S\subset\R^d$ be a bounded set. If $\varphi=(\varphi_x)_{x\in S}$ is a centered real Gaussian process with covariance matrix $\nicefrac{C}{2}\in M_n(\R)$, then for all $k\geq 1,\,x_1,\,\ldots,\,x_k\in S$ one has
   \[
   \kappa\left(\varphi_{x_1}^2,\,\ldots,\,\varphi_{x_k}^2\right)=\f12\cyp(C)(x_1,\,\ldots,\,{x_k})\coloneqq \f12\sum_{\sigma \in \mathfrak S_n:\,\#\sigma=1}\prod_{\alpha=1}^k C(x_\alpha,\,x_{\sigma(\alpha)}).
   \]
\end{proposition}
\begin{proof}
    The desired result follows from Malyshev's formula~\cite[Equation (3.2.8)]{peccati2011wiener} 
    for generalized cumulants as a sum of products of ordinary cumulants. Note that the sum is restricted to cyclic permutations only. All Gaussian cumulants are zero except those of order two, so the result is a sum of products of covariances in the form $C(x_{i_1},\,x_{j_1})\cdots C(x_{i_k},\,x_{j_k})$. Since each value $1,...,k$ occurs once as a first index and once as a second index, $(j_1,...,j_k)$ is a permutation of $(i_1,...,i_k)$. For each cyclic permutation, there are $2^{k-1}$ distinct partitions of the $2k$ indices that satisfy the connectivity condition, all giving rise to the same contribution $C(x_{1},\,x_{\sigma(1)})\cdots C(x_{k},\,x_{\sigma(k)})/2^k$. as a consequence the joint cumulant is one-half the sum of the cyclic products.
\end{proof}
\begin{proposition}[Cumulants of complex Gaussian vector]\label{prop:cumcompl}
   Let $S\subset\R^d$ be a bounded set. If ${\bf{Z}}=(Z_x)_{x\in S}$ is a centered complex Gaussian process with covariance matrix ${C}\in M_n(\C)$, then for all  $k\geq 1,\,x_1,\,\ldots,\,x_k\in S$ one has
   \[
   \kappa\left(\left|Z_{x_1}\right|^2,\,\ldots,\,\left|Z_{x_k}\right|^2\right)=\cyp(C)(x_1,\,\ldots,\,{x_k})=\sum_{\sigma \in \mathfrak S_n:\,\#\sigma=1}\prod_{\alpha=1}^k C(x_\alpha,\,x_{\sigma(\alpha)}).
   \]
\end{proposition}
\begin{proof}
   The desired formula follows after arguing as in the proof of Proposition \ref{prop:cumreal} and using the covariance matrix $C$ instead of $\nicefrac{C}{2}$.
\end{proof}
Going back to~\eqref{eq:equal_cumulants}, we understand now that, if we want a Gaussian field with cumulants exactly equal to those of the height-one field in the limit, we would need to take a complex version of the DGFF (and possibly sum over $2d$, rather than $d$, directions). However, removing the negative sign in~\eqref{eq:equal_cumulants} seems out of reach at the moment, as the next Remark discusses.
\begin{remark}\label{rmk:cum}
      Consider two distinct random variables $\mathbf X,\mathbf Y$ defined on a common probability space on $\R^d$ such that $\kappa_{\mathbf m}(\mathbf X)=-\kappa_{\mathbf m}(\mathbf Y)$ for any $\mathbf m\in\N^d$. Thus, formally,
      \begin{equation}\label{eq:cumfunction}
          \E\left[ \exp\left({\mathbf t \cdot \mathbf X }\right) \right] = \dfrac{1}{\E\left[ \exp\left({\mathbf t\cdot \mathbf Y}\right) \right]} 
      \end{equation}
      for any $\mathbf t\in\R^d.$
      Indeed, by means of Definition~\ref{def:cum_mult} we can formally write that
      \begin{align*}
      \E\left[\exp\left({ {\mathbf t}\cdot \mathbf X}\right) \right] &= \exp\left(\displaystyle\sum_{\mathbf m\in\N^d} \kappa_{\mathbf m}(\mathbf X) \prod_{j=1}^d \frac{t_j^{m_j}}{m_j!}\right) = \exp\left(-\sum_{\mathbf m\in\N^d} \kappa_{\mathbf m}(\mathbf Y) \prod_{j=1}^d \frac{t_j^{m_j}}{m_j!}\right)\\
      &= \dfrac{1}{\E\left[ \exp\left({\mathbf t\cdot \mathbf Y} \right)\right]}.
      \end{align*}
  \end{remark}
Recalling~\eqref{eq:equal_cumulants}, we are looking for an answer to the question:
\begin{quote}
``Are there random variables that satisfy~\eqref{eq:cumfunction}?''
\end{quote}
An immediate example are the random variables $\mathbf X=\mathbf v$ almost surely and $\mathbf Y=-\mathbf v$ almost surely, with $\mathbf 0\neq \mathbf v\in\R^d$. The question is now whether there are more ``significant'' ones. As one can notice, if such variables exist, one can always construct a probability space in which they are independent, and the above informal argument would go through, yielding the same conclusion on $\mathbf X$ and $\mathbf Y$. Therefore, to find a non-trivial answer, we need to consider the question in a different setup. For this reason, we will now tackle the concept of {\it supersymmetry}, that we present in the next section.
\section{Supersymmetry (SUSY)}\label{sec:susy}
 \subsection{SUSY Gaussians} 
 For the rest of the Subsection, let  $m\coloneqq\#V$ and $n=2m$. We will start working with superspins, namely vectors living in a space $\R^{n|n}$ which is a ``hybrid'' space in which both $n$ standard, real variables and $n$ Grassmannian generators live. More precisely, we define a {\it superspin}, or supervector, as the vector
 \[
 u_i=(x_i,\,y_i,\,\xi_i,\,\bxi_i)^T,\quad i\in V
 \]
 where $x_i,\,y_i\in \R$ and $\xi_i,\,\bxi_i$ are two generators of $\GA$. A smooth superfunction $F\in C^\infty(\R^{n|n})$ can be written as
 \[
 F=\sum_{I\subseteq [n]}f_I(x_I,y_I)\xi_I\bxi_I
 \]
 where $f_I\in C^\infty(\R^n)$ are smooth functions and $\xi_I\bxi_I$ are Grassmann monomials indexed over $I$. The body $F_b$ of a superfunction $F$
is defined as the ordinary smooth function obtained by formally setting all Grassmann variables to zero:
\[
F_b=f_\emptyset(\mathbf{x},\,\mathbf{y}).
\]
The remaining part is referred to as the soul:
\[
F_s=F-F_b=\sum_{\emptyset\neq I\subseteq [n]}f_I(x_I,y_I)\xi_I\bxi_I.
\]
In superspaces, integration works in the following way. 
 \subsection{Superintegration}

The case of $\R^{0|n}$ corresponds to integration in fermionic spaces as we have already seen in Subsection~\ref{subsec:int_fer}. Namely,
 \[
 \int_{\R^{0|n}} F=\int_{\R^{0|n}}\sum_{\emptyset\neq I\subseteq [n]}f_I(x_I,y_I)\xi_I\bxi_I\coloneqq\sum_{\emptyset\neq I\subseteq [n]}f_I(x_I,y_I)\int \Diff\xi_I\bxi_I.
 \]
 On $\R^{n|n}$, Berezin measures are written as
 \[
 \De \mathbf{u}=\De(\mathbf x,\,\mathbf y,\,\boxi,\boxbxi)=\sum_{I\subseteq[n]}\De\nu_I(\mathbf x,\,\mathbf y)\xi_I\bxi_I\Diff
 \]
 where $\De\nu_I$ are measures in $\R^n$. Then
 \[
 \int F\De\mathbf u\coloneqq\sum_I\int_{\R^n}\left(\int \Diff\,F\xi_I\bxi_I\right)\De\nu_I(\mathbf x,\,\mathbf y).
 \]
 For $I=\emptyset$ we take $\De\nu_\emptyset$ to be the Lebesgue measure normalized by $\pi^{n/2}$ on $\R^n$ and $\nu_I=0$ otherwise, obtaining the Berezin-Lebesgue measure on $\R^{n|n}$. From now on, we will denote it simply as $\De \mathbf u$.
 \begin{example}\label{eq:int_one}
     If $n=2$, $a\in \R$ and $F(u)=\exp(-a (x^2+y^2)-a\xi\bxi)$ then
  \[   \int F\De u=1.\]
 \end{example}
 \begin{proof}
 Using the definition of superintegration, we obtain that
  \begin{align*}  \int \e^{-a (x^2+y^2)-a\xi\bxi}&=\int_{\R^2}  \e^{-a (x^2+y^2)}\left(\int\Diff\left(-a\xi\bxi\right)\right)\frac1{\pi} \De x \De y\\
  &=\frac{a}{\pi}\left(\int_{\R} \e^{-a x^2}\De x\right)^2=1. \end{align*}
  This also explains our choice of normalization of the Lebesgue measure on $\R^n$.
 \end{proof}
 The interesting fact is that this integral seems to be independent of $a$, and we would like to generalize this result, as well as using it as motivating example to the phenomenon of supersymmetry.

 Let $A\in M_m(\C)$ be an Hermitian positive-definite matrix. If $x_\alpha=\Re(Z_\alpha),\,y_\alpha=\Im(Z_\alpha) $, we recall the identities
 \[
 \int_{\C^m}\exp\left(-\left(\overline{\bf Z},A{\bf Z}\right)\right)\prod_{\alpha=1}^m\frac{\De x_\alpha\De y_\alpha}{\pi}= \int_{\R^{n}}\exp\left(-\left({\bf x},A{\bf x}\right)-\left({\bf y},A{\bf y}\right)\right)\prod_{\alpha=1}^m\frac{\De x_\alpha\De y_\alpha}{\pi}=\det(A)^{-1}
 \]
 and
  \[
 \int\Diff\e^{\left(\boxbxi,A\boxi\right)}=\det(A).
 \]
 Thus, by letting 
 \[
 {\Diffz }:= \prod_{\alpha=1}^m\frac{\De x_\alpha\De y_\alpha}{\pi}
 \]
 we get, as in Example~\ref{eq:int_one},
 \begin{equation}\label{eq:supint}
\int {\Diff} {\Diffz} \exp\left(-\left(\overline{\mathbf{Z}}, A {\mathbf{Z}}\right) + \left({\boxbxi}, A\boxi\right)\right) = \text{1}.
 \end{equation}
 Note that the integral \eqref{eq:supint} does not depend on the choice of the matrix $A$. 
 
 Define now the {\it super inner product} between two superspins at $i$ and $j\in [n]$ as
 \begin{equation}\label{eq:supinnprod}
    \left(u_i,u_j\right) := x_ix_j + y_iy_j + \frac{1}{2}\left(\xi_j \bxi_i +  \xi_i\bxi_j\right).
 \end{equation}
 Noting that
 \[
 \left(\boxbxi,A\boxi\right)\stackrel{\eqref{eq:AC}}{=}\f12\left(\boxbxi,A\boxi\right) -\f12\left(\boxi,A\boxbxi\right)
 \]
 we deduce that \eqref{eq:supint} becomes 
 \begin{equation}\label{eq:supgaussian}
 \int {\rm d}\mathbf{u} \ \exp\left(- \left( \mathbf{u}, A \mathbf{u} \right)\right) = 1,
 \end{equation}
 where $\mathbf{u}=(u_1,...,u_n)^T$ and 
 \[
 \left(\mathbf{u}, A \mathbf{u} \right)\coloneqq\sum_{i,j=1}^n A(i,\,j)(u_i,\,u_j)
 .\]
 The measure defined in \eqref{eq:supgaussian} is called {\it supergaussian measure}. 

  \subsection{Localisation}
  Let $\GA(\R^{n})$ be the algebra of smooth functions from $\R^{n}$ into $\GA$, which is our usual Grassmannian algebra generated by $\xi_1,...,\xi_m,\bxi_1,...,\bxi_m$. Consider the complex coordinates
  \[
  z \coloneqq x+\iu y, \qquad \bar{z} \coloneqq x-\iu y,
  \]
  and define
  \[
  \dfrac{\partial}{\partial z_i} = \dfrac{1}{2} \left( \dfrac{\partial}{\partial x_i} - \iu\dfrac{\partial}{\partial y_i} \right),
  \qquad \dfrac{\partial}{\partial \bar{z}_i} = \dfrac{1}{2} \left( \dfrac{\partial}{\partial x_i} + \iu\dfrac{\partial}{\partial y_i} \right).
  \]
  The {\it supersimmetry generator} $Q:\GA(\R^{n}) \to \GA(\R^{n})$ is then defined as
  \[
  Q := \sum_{i=1}^n \left( \xi_i\dfrac{\partial}{\partial z_i} + \bxi_i\dfrac{\partial}{\partial \bar{z}_i} 
  - 2 z_i\dfrac{\partial}{\partial \xi_i} + 2 \bar{z}_i\dfrac{\partial}{\partial \bxi_i} \right).
  \]
  A function $F\in\GA(\R^{n})$ is defined to be {\it supersymmetric} or $Q$-closed if $QF=0$, and $Q$-exact if there exists $F'\in\GA(\R^{n})$ such that $QF'=F$.  If we prefer to write the supersymmetry generator in terms of the real vectors $\bf x$ and $\bf y$, we get that
  \[
  Q =\dfrac{1}{\sqrt{\iu}}\sum_{i=1}^n Q_i\coloneqq \dfrac{1}{\sqrt{\iu}}\sum_{i=1}^n \left( \xi_i\dfrac{\partial}{\partial x_i} + \bxi_i\dfrac{\partial}{\partial y_i} 
  - 2 x_i\dfrac{\partial}{\partial \bxi_i} + 2 y_i\dfrac{\partial}{\partial \xi_i} \right).
  \]

  The fundamental property of supersimmetry is given by the following localisation theorem, whose proof can be found in \cite[Theorem 11.4.5]{bauerschmidt2019introduction}.
  \begin{theorem}\label{thm:localisation}
      Let the element $F\in\GA(\R^{n})$ be a smooth integrable supersymmetric form. Then
      \[
      \int {\rm d}\mathbf{u} \,F(\mathbf{u}) = F_b({\bf 0}),
      \]
      where $F_b$ is the body of $F$ evaluated at $\mathbf{0}$.
  \end{theorem}
We can for example prove that the super inner product is supersymmetric.
\begin{lemma}\label{lem:supersymm_inner}
For all $i,\,j\in [n]$ one has $Q(u_i,\,u_j)=0.$
\end{lemma}
\begin{proof}
    Since $Q$ formally exchanges generators as follows:
    \[
    Q x_i=\xi_i,\,Q y_i=\bar\xi_i,\,Q\bar\xi_i=-2 x_i,\,Q\xi_i=2 y_i,
    \]
    we have, up to a multiplicative constant factor $\nicefrac{1}{\sqrt{\iu}}$,
    \begin{align*}
Q(u_i,\,u_j)&=Q_i(u_i,\,u_j)+Q_j(u_i,\,u_j)=\xi _ix_j + \bar \xi_i y_j + \frac{1}{2}\left(-2\xi_j  x_i +  2 y_i\bxi_j\right)\\
&\quad +x_i\xi_j + y_i\bxi_j + \frac{1}{2}\left(2 y_j \bxi_i -2  \xi_i x_j\right)=0.\qedhere
    \end{align*}
\end{proof}
  Thanks to Theorem \ref{thm:localisation} and Lemma~\ref{lem:supersymm_inner} (for example, applying~\cite[Example 11.4.4]{bauerschmidt2019introduction}), one can show that
  \[
  \int {\rm d}\mathbf{u} \ \exp\left(- \left( \mathbf{u}, A \mathbf{u} \right)\right) \exp\left(\left( \tilde{u}, \mathbf{u} \right)\right) = 1,
  \]
  where formally $$\left( \tilde{u}, \mathbf{u} \right)\coloneqq\sum_{i=1}^n (\tilde u_i,\,u_i).$$
This is formally equivalent to
  $$
  \E\left[ \exp\left(\left( \tilde{u}, u \right)\right) \right]_{\hbox{\tiny{SUSY G}}} = 1,
  $$
  where $\E\left[\cdot\right]_{\hbox{\tiny{SUSY G}}}$ means the expectation with respect to the supergaussian measure defined in \eqref{eq:supgaussian}. So, if $\tilde{F}$ resp.\ $\tilde{G}$ is the cumulant generating function of $(\boxx, \boxy)$ resp.\ $(\boxbxi,\boxi)$ as Gaussian measures in their respective ``worlds'', then $\tilde{F}(\cdot)=-\tilde{G}(\cdot)$. Consequently, if $F$ resp.\ $G$ is the moment generating function of $(\boxx, \boxy)$ resp.\ $(\boxbxi,\boxi)$ as Gaussian measures in their respective worlds, then $F(\cdot)=1/G(\cdot)$, which accomplishes the task pointed out in Remark \ref{rmk:cum}.

  We would like to point out that supersymmetry was, in fact, not necessary to explain Remark~\ref{rmk:cum}. However we decided to introduce it here to give a glimpse into its richness and its many possible applications in probability theory.

  \section*{Acknowledgements}The authors would like to thank the organizing committee of the XXVII Brazilian School of Probability for the hospitality and the organization of the conference. We are grateful to all participants who provided us with their feedback and constructive comments. The work of SB was supported by the European Union’s Horizon 2020 research and innovation programme under the Marie Skłodowska-Curie grant agreement no.\ 101034253. SB was further supported through “Gruppo Nazionale per l’Analisi Matematica, la Probabilità e le loro Applicazioni” (GNAMPA-INdAM).
  \includegraphics[height=1em]{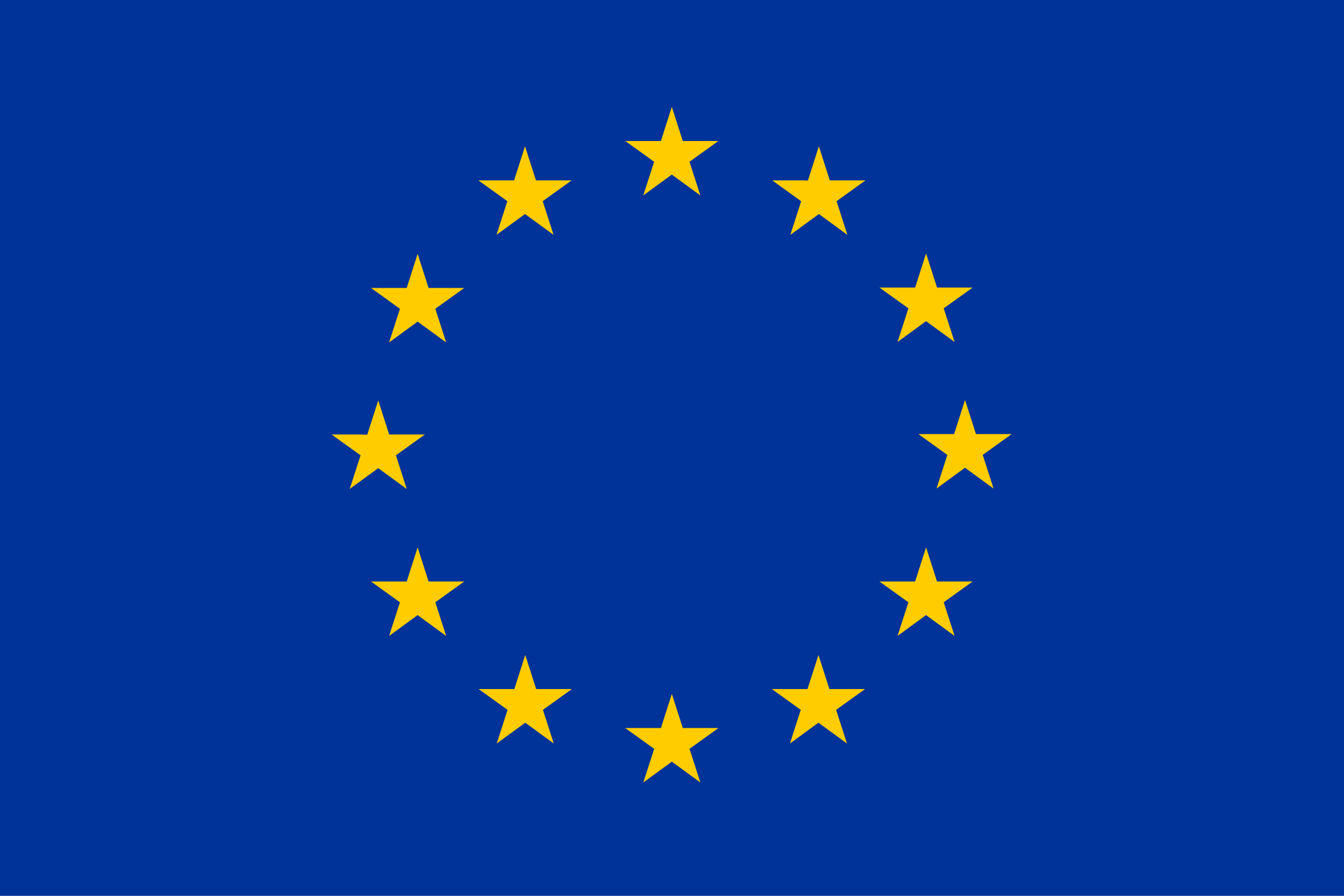}

\printbibliography
\end{document}